\documentclass[final,3p]{elsarticle}

\usepackage{amsmath,amsthm,amssymb,mathrsfs}
\usepackage{enumerate}
\usepackage{slashed}
\usepackage{txfonts,dsfont}
\usepackage{aliascnt}
\usepackage[colorlinks, linkcolor=red, citecolor=green,filecolor=magenta,urlcolor=cyan]{hyperref}

\allowdisplaybreaks

\makeatletter
\newcommand{\autorefcheckize}[1]{%
  \expandafter\let\csname @@\string#1\endcsname#1%
  \expandafter\DeclareRobustCommand\csname relax\string#1\endcsname[1]{%
    \csname @@\string#1\endcsname{##1}\wrtusdrf{##1}}%
  \expandafter\let\expandafter#1\csname relax\string#1\endcsname
}
\makeatother

\allowdisplaybreaks
\newcommand{\norm}[1]{\left\lVert#1\right\rVert}
\newcommand{\abs}[1]{\left\lvert#1\right\rvert}
\newcommand{\set}[1]{\left\{#1\right\}}
\newcommand{\kh}[1]{\left(#1\right)}%
\newcommand{\dkh}[1]{\left\{#1\right\}}%
\newcommand{\jkh}[1]{\left\langle #1\right\rangle}

\theoremstyle{plain}
\newtheorem{theorem}{Theorem}[section]

\newaliascnt{lem}{theorem}
\newtheorem{lem}[lem]{Lemma}
 \aliascntresetthe{lem}

\newaliascnt{cor}{theorem}
\newtheorem{cor}[cor]{Corollary}
 \aliascntresetthe{cor}

\newaliascnt{prop}{theorem}
\newtheorem{prop}[prop]{Proposition}
 \aliascntresetthe{prop}

\theoremstyle{remark}

\theoremstyle{definition}

\numberwithin{equation}{section}

\newcommand{\BB}{\mathbf B}
\newcommand{\HH}{\mathbf H}

\begin{document}
\hypersetup{pdftitle={A topological sphere theorem for pinched Legendrian submanifolds},pdfauthor={Qing Cui and Linlin Sun}}

\begin{frontmatter}

\title{A topological sphere theorem for pinched Legendrian submanifolds}

\author[cui]{Qing Cui\fnref{fund}}
\ead{cuiqing@swjtu.edu.cn}
\address[cui]{School of Mathematics, Southwest Jiaotong University, 611756 Chengdu, Sichuan, China}

\author[sun]{Linlin Sun\corref{cor}\fnref{fund1}}
\ead{sunll@xtu.edu.cn}
\address[sun]{School of Mathematics and Computational Science, Xiangtan University, Xiangtan 411105, Hunan, China}

\cortext[cor]{Corresponding author}
\fntext[fund]{Supported by the Natural Science Foundation of Sichuan Province (Grant No. 2025ZNSFSC0077) and the National Natural Science Foundation of China (Grant No. 12371451).}
\fntext[fund1]{Supported by the National Natural Science Foundation of China (Grant Nos. 12571055 and 12671062), the Natural Science Foundation of Hunan Province (Grant No. 2026JJ20014), the 111 Project (Grant No. D23017), and the Program for Science and Technology Innovative Research Team in Higher Educational Institutions of Hunan Province, China.}

\begin{abstract}
Let $M^n$ be a closed simply connected Legendrian submanifold of $\mathbb S^{2n+1}$, where $n\geq5$, and let $\mathbf{B}$ and $\mathbf{H}=\operatorname{tr}\mathbf{B}$ denote its second fundamental form and mean-curvature vector. We prove that $M$ is homeomorphic to $\mathbb S^n$ if
\begin{equation*}
|\mathbf{B}|^2\leq \frac{(n-1)(n+2)}n+\frac{n^2+3n-2}{2n^2}|\mathbf{H}|^2-\frac{(n-1)(n-2)}{2n^2}|\mathbf{H}|\sqrt{|\mathbf{H}|^2+4n}.
\end{equation*}
The equality can be attained by the Calabi products, so simple connectivity cannot be omitted. For $n\geq4$, we also prove a sphere theorem under an explicit bound and the weaker assumption $b_1(M;\mathbb R)=0$. The proof combines phase-rotated Lawson--Simons variations with a systolic inequality for finite two-complexes. We also give local constructions showing that the pinching condition allows negative isotropic curvature at a point.
\end{abstract}

\begin{keyword}
Legendrian submanifolds \sep sphere theorem \sep Lawson--Simons inequality \sep integral currents

\par
\MSC[2020] 53C42\sep 53C25\sep 53C40
\end{keyword}

\end{frontmatter}

\section{Introduction}

Let $\phi:M^n\to\mathbb S^{2n+1}\subset\mathbb C^{n+1}$ be a Legendrian immersion. Following Luo--Sun \cite{LuoSun}, we denote its second fundamental form by $\BB$ and use the unaveraged mean-curvature vector $\HH=\operatorname{tr}\BB$. The Legendrian condition implies that
\begin{equation*}
\sigma(X,Y,Z)=\jkh{\BB(X,Y),JZ}
\end{equation*}
is a symmetric cubic tensor. We use this symmetry to obtain topological sphere theorems from pointwise bounds on $\abs{\BB}^2$ and $\abs{\HH}$.

Luo--Sun \cite[Theorem~1.1]{LuoSun} studied closed contact stationary Legendrian submanifolds under the pinching condition determined by
\begin{equation}\label{eq:delta}
\delta_n(h)=\frac{(n-1)(n+2)}n+\frac{n^2+3n-2}{2n^2}h^2-\frac{(n-1)(n-2)}{2n^2}h\sqrt{h^2+4n},\qquad h\geq0.
\end{equation}
For $n\geq3$, their theorem gives a rigidity alternative: the immersion is minimal or is a Calabi product of a totally geodesic Legendrian immersion and a point. In a recent preprint, Luo, Sun and Yin \cite[Theorem~1.2]{LuoSunYinSharp} show that a connected closed minimal Legendrian immersion of dimension $n\geq3$ with $\abs{\BB}^2\leq\delta_n(0)$ is, up to a unitary transformation, totally geodesic or a finite covering of the minimal Calabi image.

Our first result uses the same pinching function and assumes simple connectivity in place of contact stationarity.

\begin{theorem}\label{thm:luo-sun}
Let $\phi:M^n\to\mathbb S^{2n+1}$ be a Legendrian immersion of a closed simply connected manifold of dimension $n\geq5$. If $\abs{\BB}^2\leq\delta_n(\abs{\HH})$ everywhere on $M$, then $M$ is homeomorphic to $\mathbb S^n$.
\end{theorem}

The pinching boundary is attained by the Calabi family in \cite[Appendix~A.2]{LuoSun}. With $\lambda=r_2/r_1\geq n^{-1/2}$ in their notation, its curvature norms satisfy
\begin{equation*}
\abs{\HH}=n\lambda-\lambda^{-1},\qquad \abs{\BB}^2=(3n-2)\lambda^2+\lambda^{-2}-2=\delta_n(\abs{\HH}).
\end{equation*}
Every constant $\abs{\HH}\geq0$ is attained on the complete domain $\mathbb R\times\mathbb S^{n-1}$. The immersion descends to a closed product $\mathbb S^1\times\mathbb S^{n-1}$ when $\lambda^2\in\mathbb Q$; these closed examples realize a dense set of mean-curvature norms, including zero. Thus simple connectivity cannot be omitted from Theorem~\ref{thm:luo-sun}, even at equality. This geometric attainment does not assert optimality of the bound within the simply connected class. Theorem~\ref{thm:luo-sun} gives a topological conclusion without classifying the immersion.

We prove Theorem~\ref{thm:luo-sun} by establishing a stronger pinching theorem with the assumption $b_1(M;\mathbb R)=0$. For $h\geq0$, define
\begin{equation}\label{eq:Fn-low}
\begin{aligned}
F_4(h)=&\frac92+\frac{21}{32}h^2-\frac3{32}h\sqrt{h^2+32},\\
F_5(h)=&\frac{63}{10}+\frac9{16}h^2-\frac3{80}h\sqrt{h^2+48},\\
F_6(h)=&\min\dkh{8+\frac12h^2,\,\frac{216}{25}+\frac{66}{125}h^2-\frac9{125}h\sqrt{h^2+60}}.
\end{aligned}
\end{equation}
For $n\geq7$, set
\begin{equation}\label{eq:Fn-explicit}
F_n(h)=
\begin{cases}
\begin{aligned}
&\displaystyle\frac{3(n^2-4)}{2n}+\frac{3(n+2)}{8n}h^2 -\frac{3(n-6)}{8(n-2)}h\sqrt{\frac{n-4}{n}h^2+\frac{8(n-2)^2}{n}},
\end{aligned}&0\leq h\leq\sqrt2(n-2),\\[4pt]
\displaystyle\frac{12(n-2)}n+\frac3n h^2,&h\geq\sqrt2(n-2).
\end{cases}
\end{equation}
The two branches agree at $h=\sqrt2(n-2)$. Each $F_n$ is continuous and locally Lipschitz on $[0,\infty)$.

\begin{theorem}\label{thm:sphere}
Let $\phi:M^n\to\mathbb S^{2n+1}$ be a Legendrian immersion of a closed connected manifold of dimension $n\geq4$. If $b_1(M;\mathbb R)=0$ and $\abs{\BB}^2<F_n(\abs{\HH})$ everywhere on $M$, then $M$ is homeomorphic to $\mathbb S^n$.
\end{theorem}

We show in Lemma~\ref{lem:delta-comparison} that $\delta_n(h)<F_n(h)$ for $n\geq5$. This proves Theorem~\ref{thm:luo-sun}. The function $\delta_4$ is not everywhere below $F_4$, so this comparison does not give a four-dimensional version of Theorem~\ref{thm:luo-sun}. The strict inequality in Theorem~\ref{thm:sphere} is needed for our variation argument.

For comparison, Sun--Sun \cite[Theorem~5.4]{SunSun} proved a topological sphere theorem for closed simply connected Legendrian submanifolds under
\begin{equation*}
\abs{\BB}^2\leq
\begin{cases}
6+\frac34\abs{\HH}^2,&n=4,\\[2pt]
6+\frac9{3n+2}\abs{\HH}^2,&n\geq5,
\end{cases}
\end{equation*}
with strict inequality at one point. Here their condition has been written in the unit sphere with our convention for $\HH$. For $n\geq6$, Theorem~\ref{thm:sphere} allows a larger pointwise bound and requires only $b_1(M;\mathbb R)=0$, since
\begin{equation*}
\begin{aligned}
F_6(h)\geq&\frac{162}{25}+\frac{57}{125}h^2>6+\frac9{20}h^2,\\
F_n(h)\geq&\frac{12(n-2)}n+\frac3n h^2>6+\frac9{3n+2}h^2\qquad(n\geq7).
\end{aligned}
\end{equation*}
In dimension five the bounds $F_5$ and $6+9h^2/17$ are not globally ordered. The $n=5$ case of Theorem~\ref{thm:luo-sun} is already covered by the Sun--Sun bound, because
\begin{equation*}
\delta_5(h)\leq\frac{28}{5}+\frac{13}{25}h^2<6+\frac9{17}h^2.
\end{equation*}
For $n\geq6$, the corresponding $\delta_n$ bounds are not globally ordered:
\begin{equation*}
\delta_n(0)=n+1-\frac2n>6,\qquad \lim_{h\to\infty}\frac{\delta_n(h)}{h^2}=\frac{3n-2}{n^2}<\frac9{3n+2}.
\end{equation*}

The broader setting is the theory of curvature pinching and sphere theorems. The classical sectional-curvature results of Rauch, Berger and Klingenberg \cite{Rauch,Berger,Klingenberg} were followed by differentiable pinching theorems \cite{BrendleSchoen08,BrendleSchoen09,XuGu}. Isotropic curvature and Ricci flow provide other approaches to topology and classification \cite{MicallefMoore,ChenTangZhu,BohmWilking,Brendle}. For extrinsic sphere theorems and related topological or geometric rigidity under curvature bounds, see \cite{Leung93,ShiohamaXu,CuiSun,LiXuZhao,GeTaoZhou,Vlachos}; the mean-curvature-flow approach is developed in \cite{AndrewsBaker,BakerNguyen}.

Rigidity for minimal submanifolds in spheres has its origins in the work of Simons and Chern--do Carmo--Kobayashi \cite{Simons,ChernDoCarmoKobayashi}, with further pinching results in \cite{Leung83,LiLi,ChenXu}. Related results for totally umbilic hypersurfaces and submanifolds with constant or parallel mean curvature appear in \cite{ChengNakagawa,XuRen,Xu}. Within the Legendrian class, rigidity results exploiting the symmetric cubic form include \cite{YamaguchiKonMiyahara,LuoSunYin,WuYang}. For Hamiltonian-minimal constructions, curvature inequalities and sphere theorems in Lagrangian geometry, respectively, see \cite{CastroLiUrbano,ChenPrietoWang,LiWang}. Calibrated geometry and special Lagrangian cones provide further context for the cone construction used in Section~\ref{sec:examples} \cite{HarveyLawson,Haskins}.

Integral-current vanishing arguments and stable-current methods for submanifold topology are discussed in \cite{LawsonSimons,Xin,HowardWei}. The relation between contraction by flows and topological vanishing also appears in the stochastic-flow approach of Elworthy--Rosenberg \cite{ElworthyRosenberg}; general background on stochastic flows is given in \cite{Kunita}. For the distinction between topological and smooth classifications of homotopy spheres, see \cite{KervaireMilnor}. For topology and Ricci flow in dimension three, see \cite{Moise,PerelmanSurgery,PerelmanExtinction}.

The proof uses the Lawson--Simons method \cite{LawsonSimons}. Since $b_1(M;\mathbb R)=0$, the mean-curvature form has a global potential. We use this potential to rotate the ambient constant vector fields and estimate the resulting variation traces by the symmetry of $\sigma$. These estimates give vanishing of the required integral homology groups. In degree two, paired forward and backward flows contract the area of a finite two-skeleton. The systolic inequality of Rudyak--Sabourau \cite{RudyakSabourau} then forces $\pi_1(M)$ to be free. Since $b_1(M;\mathbb R)=0$, the manifold is simply connected.

The horizontal Hopf lift preserves the second fundamental form of a simply connected Lagrangian immersion \cite{Reckziegel}. We therefore obtain the following consequence, where $\mathbb{CP}^n(4)$ has constant holomorphic sectional curvature four.

\begin{cor}\label{cor:lagrangian}
Let $\phi:M^n\to\mathbb{CP}^n(4)$ be a Lagrangian immersion of a closed simply connected manifold of dimension $n\geq4$. If $\abs{\BB}^2<F_n(\abs{\HH})$ everywhere, or if $n\geq5$ and $\abs{\BB}^2\leq\delta_n(\abs{\HH})$ everywhere, then $M$ is homeomorphic to $\mathbb S^n$.
\end{cor}

Section~\ref{sec:variation} establishes the topological criterion. Section~\ref{sec:legendrian-proof} proves the variation estimates and the sphere theorems. In Section~\ref{sec:examples}, we construct strictly $F_n$-pinched Legendrian spheres with negative isotropic curvature on a four-frame at one point. Thus our pinching condition does not imply positive isotropic curvature.

\section{A variation criterion for the sphere}\label{sec:variation}

Let $(M^n,g)$ be a closed connected Riemannian manifold. We use the curvature convention
\begin{equation*}
R(X,Y)Z=\nabla_X\nabla_YZ-\nabla_Y\nabla_XZ-\nabla_{[X,Y]}Z,\qquad R_{ijkl}=\jkh{R(e_i,e_j)e_l,e_k}.
\end{equation*}
Thus the unit sphere has $R_{ijkl}=\delta_{ik}\delta_{jl}-\delta_{il}\delta_{jk}$. We use the notation of Lawson--Simons \cite[Sections~1--3]{LawsonSimons}. Let $\varphi_t^V$ be the flow of a smooth tangent field $V$. For orthonormal vectors $e_1,\ldots,e_p\in T_xM$, put $\xi=e_1\wedge\cdots\wedge e_p$ and set
\begin{equation*}
Q_\xi(V)=\left.\frac{\mathrm d^2}{\mathrm dt^2}\right|_{t=0}\norm{\varphi_{t*}^V\xi}.
\end{equation*}
Here $\varphi_{t*}^V$ denotes the induced map on exterior powers. Since $Q_{-\xi}(V)=Q_\xi(V)$, this expression depends only on the unoriented plane $P$ represented by $\xi$. Following \cite[Section~2, Remark~4]{LawsonSimons}, write $G'_{p,q}(M)$ for the bundle of such planes, where $q=n-p$. Complete $e_1,\ldots,e_p$ to an orthonormal frame by $e_{p+1},\ldots,e_n$. Latin indices range from $1$ to $p$ and Greek indices from $p+1$ to $n$. The pointwise form of the second variation formula in \cite[Section~2, Theorem~1 and Remark~2]{LawsonSimons} is
\begin{equation}\label{eq:variation-pointwise}
Q_\xi(V)=\sum_i\jkh{\nabla_{e_i}\nabla_VV,e_i}+\sum_i\jkh{R(V,e_i)V,e_i}+\kh{\sum_i\jkh{\nabla_{e_i}V,e_i}}^2-\sum_{i,j}\jkh{\nabla_{e_i}V,e_j}\jkh{\nabla_{e_j}V,e_i}+\sum_{i,\alpha}\jkh{\nabla_{e_i}V,e_\alpha}^2.
\end{equation}
For a rectifiable $p$-current $\mathscr S$, let $\mathbf M(\mathscr S)$, $\norm{\mathscr S}$ and $\vec{\mathscr S}$ denote its mass, mass measure and unit oriented tangent $p$-vector, respectively. The area formula gives
\begin{equation*}
Q_{\mathscr S}(V):=\left.\frac{\mathrm d^2}{\mathrm dt^2}\right|_{t=0}\mathbf M\kh{\varphi_{t*}^V\mathscr S}=\int_M Q_{\vec{\mathscr S}(x)}(V)\,\mathrm d\norm{\mathscr S}(x).
\end{equation*}
We also apply \eqref{eq:variation-pointwise} to the faces of finite two-complexes.

\begin{prop}\label{prop:LS-sphere}
Suppose that $n\geq4$ and $b_1(M;\mathbb R)=0$. For each $2\leq p\leq\lfloor n/2\rfloor$, assume that there are finitely many smooth tangent fields $V_1,\ldots,V_m$ such that
\begin{equation}\label{eq:negative-variation}
\sum_{A=1}^m Q_\xi(V_A)<0
\end{equation}
for every $P\in G'_{p,n-p}(M)$, where $\xi$ is either unit simple vector representing $P$. The family may depend on $p$. Then $M$ is homeomorphic to $\mathbb S^n$.
\end{prop}

\begin{proof}
Fix $p$ in the stated range. If $H_p(M;\mathbb Z)\ne0$, the Federer--Fleming theorem \cite{FedererFleming} gives a nonzero mass-minimizing integral cycle $\mathscr S$. Each flow is isotopic to the identity, so the area formula and stability give
\begin{equation*}
0\leq\sum_A Q_{\mathscr S}(V_A)=\int_M\sum_A Q_{\vec{\mathscr S}(x)}(V_A)\,\mathrm d\norm{\mathscr S}(x)<0.
\end{equation*}
Consequently,
\begin{equation}\label{eq:homology-vanishing}
H_p(M;\mathbb Z)=0\qquad\kh{2\leq p\leq\lfloor n/2\rfloor}.
\end{equation}

We next show that $\pi_1(M)$ is free. Apply \eqref{eq:negative-variation} with $p=2$. Compactness of $G'_{2,n-2}(M)$ gives a constant $\kappa>0$ such that $\sum_A Q_\xi(V_A)\leq-\kappa$. Taylor's formula yields, uniformly over unit simple two-vectors $\xi$,
\begin{equation*}
\sum_A\kh{\norm{\varphi_{t*}^{V_A}\xi}+\norm{\varphi_{-t*}^{V_A}\xi}}=2m+t^2\sum_A Q_\xi(V_A)+O(t^4).
\end{equation*}
Choose $\tau>0$ so that the remainder is at most $\kappa\tau^2/2$ and $\rho=1-\kappa\tau^2/(4m)\in(0,1)$. For every embedded finite piecewise smooth two-complex $K\subset M$, integration over its faces gives
\begin{equation*}
\sum_{A=1}^m\sum_{\varepsilon=\pm1}\operatorname{Area}\kh{\varphi_{\varepsilon\tau}^{V_A}(K)}\leq\kh{2m-\frac{\kappa\tau^2}{2}}\operatorname{Area}(K).
\end{equation*}
One of these diffeomorphisms decreases the area by a factor at most $\rho$.

Let $K_0$ be the two-skeleton of a finite smooth triangulation of $M$. The inclusion induces an isomorphism $\pi_1(K_0)\cong\pi_1(M)$. Iterating the preceding choice gives diffeomorphisms $\psi_j$, isotopic to the identity, such that
\begin{equation}\label{eq:area-decay}
\operatorname{Area}\kh{\psi_j(K_0)}\leq\rho^j\operatorname{Area}(K_0)\longrightarrow0.
\end{equation}
Suppose that $\pi_1(M)$ is not free. Every noncontractible loop in $\psi_j(K_0)$ is noncontractible in $M$, and its length is unchanged by viewing it as a curve in $M$. Thus
\begin{equation*}
\operatorname{sys}_{\pi_1}\kh{\psi_j(K_0)}\geq\operatorname{sys}_{\pi_1}(M,g)>0,
\end{equation*}
where $\operatorname{sys}_{\pi_1}$ is the infimum of the lengths of noncontractible loops. By the systolic inequality for finite piecewise Riemannian two-complexes with nonfree fundamental group \cite[Theorem~8.1]{RudyakSabourau},
\begin{equation*}
\operatorname{Area}\kh{\psi_j(K_0)}\geq\frac{\pi}{16}\operatorname{sys}_{\pi_1}\kh{\psi_j(K_0)}^2\geq\frac{\pi}{16}\operatorname{sys}_{\pi_1}(M,g)^2>0.
\end{equation*}
This contradicts \eqref{eq:area-decay}. Hence $\pi_1(M)$ is free. Its rank equals $b_1(M;\mathbb R)=0$, so $M$ is simply connected and orientable.

Combining \eqref{eq:homology-vanishing} with the integral universal coefficient theorem and Poincar\'e duality gives
\begin{equation*}
H_j(M;\mathbb Z)=0\quad(1\leq j<n),\qquad H_n(M;\mathbb Z)\cong\mathbb Z.
\end{equation*}
The Hurewicz theorem, applied successively, yields $\pi_j(M)=0$ for $1\leq j<n$ and $\pi_n(M)\cong\mathbb Z$. A map $\mathbb S^n\to M$ representing a generator is a homology equivalence and hence a homotopy equivalence \cite{Hatcher}. Freedman's theorem when $n=4$ and Smale's theorem when $n\geq5$ complete the proof \cite{Freedman,Smale}.
\end{proof}

\section{Sphere theorems}\label{sec:legendrian-proof}

\subsection{Rotated variation fields}

Let $\phi:M^n\to\mathbb S^{2n+1}\subset\mathbb C^{n+1}$ be Legendrian for the contact form $\bar\alpha_z(V)=\frac12\jkh{Jz,V}$, as in \cite[Section~2]{LuoSun}. We identify tangent vectors with their images under $\mathrm d\phi$. The normal bundle in the sphere is $JTM\oplus\mathbb R J\phi$, and $\jkh{\BB(X,Y),J\phi}=0$. Following \cite[Section~2]{LuoSun}, write in a local orthonormal frame
\begin{equation*}
\sigma_{ijk}=\jkh{\BB(e_i,e_j),Je_k},\qquad \mu(X)=\jkh{\HH,JX},\qquad \mu_k=\jkh{\HH,Je_k}=\sum_j\sigma_{jjk}.
\end{equation*}
Then $\sigma$ is symmetric, $\abs{\sigma}=\abs{\BB}$, and $\abs{\mu}=\abs{\HH}$. The normal connection satisfies
\begin{equation*}
\nabla_X^\perp(JY)=J\nabla_XY-\jkh{X,Y}J\phi,
\end{equation*}
and hence
\begin{equation*}
(\nabla_X\sigma)(Y,Z,W)=\jkh{(\nabla_X^\perp\BB)(Y,Z),JW}.
\end{equation*}
Codazzi and the symmetry of $\sigma$ imply that $\nabla\sigma$ is symmetric. Taking a trace gives $\mathrm d\mu=0$. In particular, if $b_1(M;\mathbb R)=0$, then $\mu=\mathrm d\vartheta$ for a smooth function $\vartheta$.

Fix a $p$-plane $P$, choose a unit simple $p$-vector $\xi$ representing it, and put $q=n-p$, $Q=P^\perp$. Latin indices range over $P$, Greek indices over $Q$, and $l$ over the full tangent space. For a symmetric cubic tensor $T$, set
\begin{equation*}
\mathcal L_p^0(T;P)=2\sum_{i,\alpha,l}T_{i\alpha l}^2-\sum_l\kh{\sum_iT_{iil}}\kh{\sum_\alpha T_{\alpha\alpha l}}.
\end{equation*}
This expression satisfies
\begin{equation}\label{eq:classical-complement}
\mathcal L_p^0(T;P)=\mathcal L_{n-p}^0(T;P^\perp).
\end{equation}
For the cubic form $\sigma$, abbreviate $\mathcal L_p^0(\sigma;P)$ to $\mathcal L_p^0(P)$ and define, for a one-form $v$,
\begin{equation}\label{eq:rotated-trace}
\mathcal L_p^v(P)=\mathcal L_p^0(P)+q\abs{v_P}^2-2\sum_{i,\alpha}v_i\sigma_{i\alpha\alpha}.
\end{equation}
Here $v_P$ denotes the restriction of $v$ to $P$.

\begin{lem}\label{lem:rotated-fields}
Let $a_1,\ldots,a_{2n+2}$ be a real orthonormal basis of $\mathbb C^{n+1}$, and let $\eta$ be a smooth real function on $M$. The tangent fields
\begin{equation}\label{eq:rotation-fields}
X_A=\kh{e^{\mathrm i\eta}a_A}^{\top}\qquad(1\leq A\leq2n+2)
\end{equation}
satisfy
\begin{equation}\label{eq:rotation-drift}
\sum_A Q_\xi(X_A)=\mathcal L_p^{\mathrm d\eta}(P)-pq
\end{equation}
for every unit simple $p$-vector $\xi$ representing a tangent $p$-plane $P$.
\end{lem}

\begin{proof}
Put $v=\mathrm d\eta$. Differentiating the tangential projection gives
\begin{equation}\label{eq:rotation-derivative}
\jkh{\nabla_{e_i}X_A,e_j}=\sum_l\kh{\sigma_{ijl}-v_i\delta_{jl}}\jkh{e^{\mathrm i\eta}a_A,Je_l}-\delta_{ij}\jkh{e^{\mathrm i\eta}a_A,\phi}.
\end{equation}
Since $TM$, $JTM$, and $\operatorname{span}_{\mathbb R}\set{\phi,J\phi}$ are mutually orthogonal,
\begin{equation}\label{eq:rotation-covariance}
\sum_A\jkh{X_A,X}\jkh{X_A,Y}=\jkh{X,Y},\qquad \sum_A\nabla_{X_A}X_A=0.
\end{equation}
Thus the acceleration terms in \eqref{eq:variation-pointwise} sum to zero, and the curvature terms sum to $-\sum_i\operatorname{Ric}(e_i,e_i)$.

First set $v=0$ in \eqref{eq:rotation-derivative}. The remaining terms in \eqref{eq:variation-pointwise}, summed over $A$, are
\begin{equation*}
\sum_l\kh{\sum_i\sigma_{iil}}^2-\sum_{i,j,l}\sigma_{ijl}^2+\sum_{i,\alpha,l}\sigma_{i\alpha l}^2+p^2-p.
\end{equation*}
By the Gauss equation,
\begin{equation*}
\sum_i\operatorname{Ric}(e_i,e_i)=p(n-1)+\sum_{i,l}\mu_l\sigma_{iil}-\sum_{i,j,l}\sigma_{ijl}^2-\sum_{i,\alpha,l}\sigma_{i\alpha l}^2.
\end{equation*}
Their difference is $\mathcal L_p^0(P)-pq$. For general $v$, the changes in the two tangential terms are
\begin{equation*}
-2\sum_i v_i\sum_j\sigma_{jji}+\abs{v_P}^2,\qquad 2\sum_i v_i\sum_j\sigma_{jji}-\abs{v_P}^2.
\end{equation*}
They cancel. The transverse term changes by $q\abs{v_P}^2-2\sum_{i,\alpha}v_i\sigma_{i\alpha\alpha}$, proving \eqref{eq:rotation-drift}.
\end{proof}

\subsection{The cubic estimate}

For $2\leq p\leq\lfloor n/2\rfloor$, choose
\begin{equation*}
t_p=\min\dkh{\frac{q-p+2}{2(n+2)},\frac1q}.
\end{equation*}
When $\mu=\mathrm d\vartheta$, the choice $\eta=t_p\vartheta$ gives $\mathrm d\eta=t_p\mu$ in Lemma~\ref{lem:rotated-fields}. These fields remain smooth where $\HH=0$.

For a symmetric cubic tensor $\sigma$ with trace $\mu$, write
\begin{equation*}
\mathring\sigma_{ijk}=\sigma_{ijk}-\frac{\mu_i\delta_{jk}+\mu_j\delta_{ik}+\mu_k\delta_{ij}}{n+2},\qquad S=\abs{\sigma}^2,\quad h=\abs{\mu},\quad r=\abs{\mathring\sigma}.
\end{equation*}
Then $r^2=S-3h^2/(n+2)$.

\begin{lem}\label{lem:cubic-bound}
For every symmetric cubic tensor $\sigma$ as above and every $p$-plane $P$,
\begin{equation*}
\mathcal L_p^{t_p\mu}(P)\leq\mathcal E_{p,q}(S,h),
\end{equation*}
where
\begin{equation}\label{eq:envelope}
\begin{aligned}
&\mathcal E_{p,q}(S,h)=\frac{(p+2)(q+2)}{3(n+2)}r^2-\frac{pq}{(n+2)^2}h^2\\
&+\max_{0\leq z\leq1}\left\{qt_p\kh{t_p-\frac2{n+2}}zh^2 +\frac{rh}{n+2}\left[\frac{q(p+2)}{3(n+2)}\kh{q-p+2-2(n+2)t_p}^2z +\frac{p(q+2)}{3(n+2)}(q-p-2)^2(1-z)\right]^{1/2}\right\}.
\end{aligned}
\end{equation}
In particular, $\mathcal E_{p,q}(S,0)=\frac{(p+2)(q+2)}{3(n+2)}S$.
\end{lem}

\begin{proof}
Set
\begin{equation*}
a_l=\sum_i\mathring\sigma_{iil}=-\sum_\alpha\mathring\sigma_{\alpha\alpha l}.
\end{equation*}
The four blocks of $\mathring\sigma$ have the orthogonal decompositions
\begin{equation*}
\begin{aligned}
\mathring\sigma_{ijk}=&\rho_{ijk}+\frac{a_i\delta_{jk}+a_j\delta_{ik}+a_k\delta_{ij}}{p+2},\\
\mathring\sigma_{ij\alpha}=&\rho_{ij\alpha}+\frac{a_\alpha}{p}\delta_{ij},\\
\mathring\sigma_{i\alpha\gamma}=&\rho_{i\alpha\gamma}-\frac{a_i}{q}\delta_{\alpha\gamma},\\
\mathring\sigma_{\alpha\gamma\delta}=&\rho_{\alpha\gamma\delta}-\frac{a_\alpha\delta_{\gamma\delta}+a_\gamma\delta_{\alpha\delta}+a_\delta\delta_{\alpha\gamma}}{q+2}.
\end{aligned}
\end{equation*}
Each $\rho$ is trace free in the indices from the same subspace. Since each mixed block occurs three times in the full norm,
\begin{equation}\label{eq:cubic-norm}
r^2=\frac{3(n+2)}{q(p+2)}\abs{a_P}^2+\frac{3(n+2)}{p(q+2)}\abs{a_Q}^2+\abs{\rho_{PPP}}^2+3\abs{\rho_{PPQ}}^2+3\abs{\rho_{PQQ}}^2+\abs{\rho_{QQQ}}^2.
\end{equation}
Substitution in \eqref{eq:rotated-trace} gives
\begin{equation}\label{eq:cubic-identity}
\begin{aligned}
\mathcal L_p^{t_p\mu}(P)=&\frac{(p+2)(q+2)}{3(n+2)}r^2-\frac{pq}{(n+2)^2}h^2+qt_p\kh{t_p-\frac2{n+2}}\abs{\mu_P}^2-\frac{(p+2)(q+2)}{3(n+2)}\kh{\abs{\rho_{PPP}}^2+\abs{\rho_{QQQ}}^2}\\
&-\kh{\frac{q-p+2}{n+2}-2t_p}\jkh{a_P,\mu_P}-\frac{q-p-2}{n+2}\jkh{a_Q,\mu_Q}-\frac{pq}{n+2}\kh{\abs{\rho_{PPQ}}^2+\abs{\rho_{PQQ}}^2}.
\end{aligned}
\end{equation}
If $h=0$, this identity proves the assertion. If $h>0$, put $z=\abs{\mu_P}^2/h^2$. Equation~\eqref{eq:cubic-norm} gives the joint bound
\begin{equation*}
\frac{3(n+2)}{q(p+2)}\abs{a_P}^2+\frac{3(n+2)}{p(q+2)}\abs{a_Q}^2\leq r^2.
\end{equation*}
Weighted Cauchy--Schwarz therefore yields
\begin{equation*}
\begin{aligned}
&\abs{\kh{\frac{q-p+2}{n+2}-2t_p}\jkh{a_P,\mu_P}+\frac{q-p-2}{n+2}\jkh{a_Q,\mu_Q}}\\
\leq&\frac{rh}{n+2}\left[\frac{q(p+2)}{3(n+2)}\kh{q-p+2-2(n+2)t_p}^2z+\frac{p(q+2)}{3(n+2)}(q-p-2)^2(1-z)\right]^{1/2}.
\end{aligned}
\end{equation*}
Drop the last two nonpositive terms in \eqref{eq:cubic-identity} and maximize over $z\in[0,1]$.
\end{proof}

\subsection{The pinching threshold}

\begin{lem}\label{lem:controlling-degree}
Let $n\geq4$, $2\leq p\leq\lfloor n/2\rfloor$, $q=n-p$ and $S\geq3h^2/(n+2)$. If $S\leq F_n(h)$, then $\mathcal E_{p,q}(S,h)\leq pq$.
\end{lem}

\begin{proof}
For fixed $h$ and $S_2>S_1\geq3h^2/(n+2)$, the square-root terms in \eqref{eq:envelope} are nondecreasing in $S$. Hence
\begin{equation}\label{eq:envelope-monotonicity}
\mathcal E_{p,q}(S_2,h)-\mathcal E_{p,q}(S_1,h)\geq\frac{(p+2)(q+2)}{3(n+2)}(S_2-S_1)>0.
\end{equation}
For $h=0$, we have
\begin{equation*}
\frac{\mathcal E_{p,q}(S,0)}{pq}=\kh{\frac1{3(n+2)}+\frac2{3pq}}S\leq\kh{\frac1{3(n+2)}+\frac1{3(n-2)}}F_n(0)=1.
\end{equation*}
For $h>0$, it suffices to take $S=F_n(h)$. Set
\begin{equation*}
u=\frac{(n+2)r}{h},\qquad D_p(u)=\frac{\mathcal E_{p,q}(S,h)}{pq\kh{h/(n+2)}^2}\quad(u\geq0).
\end{equation*}

If $q-p\leq2$, then $t_p=(q-p+2)/(2(n+2))$. Both terms in the maximand of \eqref{eq:envelope} are nonincreasing in $z$. The maximum is attained at $z=0$, and
\begin{equation}\label{eq:balanced-envelope}
\mathcal E_{p,q}=\frac{(p+2)(q+2)}{3(n+2)}r^2-\frac{pq}{(n+2)^2}h^2+\frac{\abs{q-p-2}}{n+2}\sqrt{\frac{p(q+2)}{3(n+2)}}\,rh.
\end{equation}

If $q-p>2$, then $t_p=1/q$. With
\begin{equation*}
y=\sqrt{\frac{p(q+2)}{3(n+2)}+\frac{2(q+2)}{3q}z},
\end{equation*}
equation~\eqref{eq:envelope} becomes
\begin{equation*}
\frac{(n+2)^2}{h^2}\mathcal E_{p,q}=\frac{(p+2)(q+2)}{3(n+2)}u^2-pq+(q-p-2)\max_{0\leq z\leq1}\dkh{-\frac{3(n+2)}{2(q+2)}y^2+uy+\frac p2}.
\end{equation*}
The vertex $y=u(q+2)/(3(n+2))$ enters and leaves the allowed interval at $U_0$ and $U_1$, where
\begin{equation*}
U_0^2=\frac{3(n+2)p}{q+2},\qquad U_1^2=\frac{3(n+2)(p+2)}q.
\end{equation*}
Substitution gives $\mathcal E_{p,q}=0$ at $u=U_0$; hence $\mathcal E_{p,q}\leq0$ for $u\leq U_0$. For $u\geq U_0$,
\begin{equation}\label{eq:normalized-envelope}
D_p(u)=
\begin{cases}
\displaystyle D_p^{(1)}(u)=\frac{u^2}{6p}+\frac{u^2}{3pq}-\frac{n+2}{2q},&U_0\leq u\leq U_1,\\[6pt]
 \displaystyle D_p^{(2)}(u)=a_pu^2+b_pu-1-\frac{(n+2)(q-p-2)}{pq^2},&u\geq U_1,
\end{cases}
\end{equation}
where
\begin{equation*}
a_p=\frac1{3(n+2)}+\frac2{3pq},\qquad b_p=\frac1{\sqrt{3(n+2)}}\,\frac{q-p-2}{q}\,\frac{q+2}{\sqrt q}\,\frac{\sqrt{p+2}}p.
\end{equation*}
When $q-p=2$, both expressions reduce to $a_pu^2-1$, so the same formulas apply.

Now suppose $n\geq7$. Fix $n,u$ and allow $p\in[2,(n-2)/2]$ to be real, with $q=n-p$. Since $pq$ increases and $q$ decreases, $\partial_pD_p^{(1)}<0$ and $a_p'<0$. For the three factors in $b_p$,
\begin{equation*}
\frac{\mathrm{d}}{\mathrm{d}p}\frac{q-p-2}{q}=-\frac{n+2}{q^2},\qquad \frac{\mathrm{d}}{\mathrm{d}p}\frac{q+2}{\sqrt q}=-\frac{q-2}{2q^{3/2}},\qquad \frac{\mathrm{d}}{\mathrm{d}p}\frac{\sqrt{p+2}}p=-\frac{p+4}{2p^2\sqrt{p+2}}.
\end{equation*}
Since $q\geq p+2\geq4$, it follows that $b_p'<0$. Moreover,
\begin{equation}\label{eq:branch-difference}
D_p^{(2)}(u)-D_p^{(1)}(u)=-\frac{(q-p-2)(q+2)}{6(n+2)pq}(u-U_1)^2.
\end{equation}
The branches and their $u$ derivatives agree at $U_1$. Differentiating at this moving endpoint gives
\begin{equation*}
\partial_pD_p^{(2)}(U_1)=\partial_pD_p^{(1)}(U_1)<0.
\end{equation*}
Also,
\begin{equation*}
\partial_u\partial_pD_p^{(2)}=2a_p'u+b_p'<0.
\end{equation*}
Consequently, $\partial_pD_p^{(2)}<0$ for $u\geq U_1$. Since $U_0$ increases with $p$, the function $\max\dkh{D_p(u),0}$ is nonincreasing in $p$. In particular,
\begin{equation}\label{eq:degree-domination}
\frac{\mathcal E_{p,q}}{pq}\leq\max\dkh{\frac{\mathcal E_{2,n-2}}{2(n-2)},0}\qquad(q-p\geq2).
\end{equation}

It remains to compare the degrees with $q-p=0,1$. By \eqref{eq:balanced-envelope},
\begin{equation*}
D_p(u)=a_pu^2-1+b_{0,p}u,\qquad b_{0,p}=\frac{\abs{q-p-2}}q\sqrt{\frac{q+2}{3(n+2)p}}.
\end{equation*}
Taking $z=0$ in the degree-two maximum gives $D_2(u)\geq a_2u^2-1+b_{0,2}u$. We have $a_p\leq a_2$ and
\begin{equation*}
b_{0,2}^2=\frac16\,\frac n{n+2}\kh{\frac{n-6}{n-2}}^2,
\end{equation*}
which increases with $n\geq7$. For $n=2m+1$ and $n=2m$, respectively, with $p=m$, the remaining coefficients satisfy
\begin{equation*}
b_{0,p}^2=\frac{1+3/m}{3(2m+3)(m+1)^2},\qquad b_{0,p}^2=\frac{2}{3m^3}\kh{1+\frac1{m+1}}.
\end{equation*}
Both expressions decrease with $m$. Their ratios to $b_{0,2}^2$ at $n=7,8$ are $25/28$ and $27/32$, respectively. Hence $b_{0,p}<b_{0,2}$ in every remaining degree. Together with \eqref{eq:degree-domination}, this proves
\begin{equation}\label{eq:degree-two-suffices}
\mathcal E_{2,n-2}\leq2(n-2)\quad\Longrightarrow\quad\mathcal E_{p,n-p}\leq p(n-p)\quad(n\geq7).
\end{equation}

It remains to solve $\mathcal E_{p,q}=pq$. When $q-p\leq2$, equation~\eqref{eq:balanced-envelope} gives
\begin{equation*}
r=\frac1{2(p+2)}\sqrt{\frac{3p(n+2)}{q+2}}\kh{\sqrt{h^2+4q(p+2)}-\frac{\abs{q-p-2}}{n+2}h},
\end{equation*}
where we used $(q-p-2)^2+4q(p+2)=(n+2)^2$. Substituting $(p,q)=(2,2),(2,3),(3,3)$ and then using $S=r^2+3h^2/(n+2)$ gives $F_4$, $F_5$ and the second expression in the definition of $F_6$, respectively. For $(p,q)=(2,4)$, the root is $S=8+h^2/2$. This proves the lemma in dimensions four, five and six.

For $n\geq7$, take $p=2$ and $q=n-2\geq5$. At $u=U_1$, one has $\mathcal E_{2,q}=h^2/q$. Since the expression is increasing in $r$, its level $2q$ lies in the first branch of \eqref{eq:normalized-envelope} exactly when $h\geq\sqrt2\,q$. This branch gives
\begin{equation*}
S=\frac{12q+3h^2}{q+2}.
\end{equation*}
For $0<h\leq\sqrt2\,q$, the second branch gives
\begin{equation}\label{eq:p2-threshold-root}
r=\sqrt{\frac{3(n+2)}{16q}}\kh{\sqrt{\frac{q-2}{q+2}h^2+\frac{8q^2}{q+2}}-\frac{q-4}{n+2}h}.
\end{equation}
After substituting $S=r^2+3h^2/(n+2)$, these are precisely the two expressions in \eqref{eq:Fn-explicit}. The assertion now follows from \eqref{eq:degree-two-suffices}.
\end{proof}

\begin{proof}[Proof of Theorem~\ref{thm:sphere}]
Since $b_1(M;\mathbb R)=0$, write $\mu=\mathrm d\vartheta$. For each $2\leq p\leq\lfloor n/2\rfloor$, use the fields in Lemma~\ref{lem:rotated-fields} with $\eta=t_p\vartheta$. Put $S=\abs{\BB}^2$ and $h=\abs{\HH}$. Lemmas~\ref{lem:cubic-bound} and \ref{lem:controlling-degree}, together with \eqref{eq:envelope-monotonicity}, give
\begin{equation*}
\sum_A Q_\xi(X_A)=\mathcal L_p^{t_p\mu}(P)-pq
\leq\mathcal E_{p,q}(S,h)-pq
\leq-\frac{(p+2)(q+2)}{3(n+2)}\kh{F_n(h)-S}<0.
\end{equation*}
Proposition~\ref{prop:LS-sphere} completes the proof.
\end{proof}

\subsection{Comparison with the Luo--Sun bound}

\begin{lem}\label{lem:delta-comparison}
For every $n\geq5$ and $h\geq0$,
\begin{equation*}
F_n(h)>\delta_n(h).
\end{equation*}
\end{lem}

\begin{proof}
At $h=0$, the assertion follows from $F_n(0)-\delta_n(0)=(n+2)(n-4)/(2n)>0$. We therefore assume $h>0$. For $x\geq0$, put $R_a=\sqrt{x(x+a)}$. Rationalization gives
\begin{equation*}
R_a\geq x,\qquad 0\leq R_b-R_a\leq\frac{b-a}{2}\qquad(0<a<b).
\end{equation*}
Applying these inequalities to \eqref{eq:Fn-low} and \eqref{eq:delta}, with $x=h^2$, gives
\begin{equation*}
\begin{aligned}
F_5(h)-\delta_5(h)\geq&\frac7{40}+\frac{x}{200}>0,\\
8+\frac{x}{2}-\delta_6(h)\geq&\frac43+\frac{x}{18}>0,\\
\frac{216}{25}+\frac{66}{125}x-\frac9{125}R_{60} -\delta_6(h) \geq&\frac{254}{375}+\frac{13}{1125}x>0.
\end{aligned}
\end{equation*}
In the first and third estimates, we used $R_{48}-R_{20}\leq14$ and $R_{60}-R_{24}\leq18$, respectively. This proves the assertion for $n=5,6$.

Let $n\geq7$ and put $q=n-2$. Make the substitution
\begin{equation*}
a=\frac{\sqrt{h^2+4n}-h}{2},\qquad x=a^2.
\end{equation*}
Then $0<x\leq n$, and $a^2+ha=n$ gives
\begin{equation}\label{eq:delta-parameter}
h=\frac{n-x}{\sqrt x},\qquad \delta_n(h)=x+\frac{3n-2}{x}-2,\qquad \delta_n(h)-\frac{3h^2}{n+2} =\frac{n-1}{n+2}\frac{(x+2)^2}{x}.
\end{equation}

First suppose $h\leq\sqrt2\,q$. Comparing the positive square roots in \eqref{eq:p2-threshold-root} and \eqref{eq:delta-parameter}, we obtain $\delta_n(h)<F_n(h)$ if and only if $\Psi(x)>0$, where
\begin{equation*}
\Psi(x)=(n+2)\sqrt{\frac{n-4}{n}(n-x)^2+\frac{8q^2}{n}x} -\kappa(x+2)-(n-6)(n-x),\qquad \kappa=4\sqrt{\frac{(n-1)q}{3}}.
\end{equation*}
The radicand is
\begin{equation*}
\frac{n-4}{n}(n+x)^2+\frac{4(q^2+4)}n x.
\end{equation*}
The function $\sqrt{A(n+x)^2+Bx}$ is concave for $A,B>0$, since its second derivative is
\begin{equation*}
-\frac{B(4An+B)}{4\dkh{A(n+x)^2+Bx}^{3/2}}<0.
\end{equation*}
Thus $\Psi$ is concave on $[0,n]$. At its right endpoint,
\begin{equation*}
\Psi(n)=(n+2)(2\sqrt2\,q-\kappa)>0, \qquad \frac{\kappa^2}{8q^2}=\frac{2(n-1)}{3(n-2)}<1.
\end{equation*}
At the left endpoint, $\sqrt{n(n-4)}>n-3$ gives
\begin{equation*}
\Psi(0)>5n-6-2\kappa>0.
\end{equation*}
The last inequality follows from
\begin{equation*}
2\kappa< \frac8{\sqrt3}\kh{n-\frac32}<5n-6.
\end{equation*}
Concavity now yields $\Psi(x)>0$ throughout $[0,n]$.

If $h\geq\sqrt2\,q$, then $\sqrt2(n-2)>n-1$, and equation \eqref{eq:delta-parameter} implies $x<1$. Hence
\begin{equation*}
\delta_n(h)-\frac3n h^2 =4+\frac{n-3}{n}x-\frac2x <3-\frac3n<\frac{12(n-2)}n.
\end{equation*}
The second branch of \eqref{eq:Fn-explicit} gives the result.
\end{proof}

\begin{proof}[Proof of Theorem~\ref{thm:luo-sun}]
For $n\geq5$, Lemma~\ref{lem:delta-comparison} gives $\abs{\BB}^2\leq\delta_n(\abs{\HH})<F_n(\abs{\HH})$. Theorem~\ref{thm:sphere} applies.
\end{proof}

\begin{proof}[Proof of Corollary~\ref{cor:lagrangian}]
The pullback of the Hopf connection by $\phi$ is flat, because its curvature is a constant multiple of the K\"ahler form. Since $M$ is simply connected, $\phi$ has a global horizontal lift $F:M\to\mathbb S^{2n+1}$. This lift is isometric and Legendrian. Its second fundamental form has no vertical component:
\begin{equation*}
\jkh{ \BB^F(X,Y),JF} =-\jkh{\mathrm{d} F(Y),J\mathrm{d} F(X)}=0.
\end{equation*}
Its horizontal second fundamental form projects to that of $\phi$. Consequently $\abs{\BB^F}=\abs{\BB^\phi}$ and $\abs{\HH^F}=\abs{\HH^\phi}$. The assertion follows from Theorem~\ref{thm:sphere} and Theorem~\ref{thm:luo-sun}.
\end{proof}

\section{Pinching and isotropic curvature}\label{sec:examples}

We compare the pinching condition with isotropic curvature. For an orthonormal four-frame $(e_1,e_2,e_3,e_4)$, put
\begin{equation}\label{eq:isotropic}
\mathcal I(e_1,e_2,e_3,e_4)=R_{1313}+R_{1414}+R_{2323}+R_{2424}-2R_{1234}.
\end{equation}

\begin{theorem}\label{thm:separation}
For every $n\geq5$, there is a Legendrian embedding $\phi:\mathbb S^n\longrightarrow \mathbb S^{2n+1}$ such that, at every point and for every tangent $p$-plane $P$,
\begin{equation*}
\abs{\BB}^2<F_n(\abs{\HH}),\qquad \mathcal L_p^0(P)<p(n-p) \quad(2\leq p\leq n-2).
\end{equation*}
Moreover, $\mathcal I<0$ on an orthonormal four-frame at one point. For $n\geq6$, the embedding can also be chosen to satisfy the Lawson--Simons inequalities in degrees $1$ and $n-1$.
\end{theorem}

The construction uses the following local realization of a trace-free cubic tensor.

\begin{lem}\label{lem:cubic-localization}
Let $T\in\operatorname{Sym}^3(\mathbb R^n)^*$ be trace free. For all sufficiently small $\varepsilon>0$, there is a Legendrian embedding $F_\varepsilon:\mathbb S^n\to \mathbb S^{2n+1}$, equal to the standard totally geodesic embedding outside a cap, whose cubic form is exactly $T$ at the north pole in the standard tangent frame. On the cap, an orthonormal frame and a function $0\leq\rho\leq1$ can be chosen such that
\begin{equation}\label{eq:cubic-localization}
\abs{\sigma-\rho T}\leq C\varepsilon^2,\qquad \abs{\HH}\leq C\varepsilon^2,\qquad \abs{\BB}^2\leq \abs{T}^2+C\varepsilon^2.
\end{equation}
Here $C$ depends on $n,T$ and a fixed cutoff, but not on $\varepsilon$.
\end{lem}

\begin{proof}
Set $P(w)=T(w,w,w)/6$. Choose a smooth function $\chi$ with $0\leq\chi\leq1$, equal to one on $(-\infty,0]$ and zero on $[1,\infty)$. For $r=\abs{w}>0$, put
\begin{equation*}
\rho(w)=\chi\kh{1+\varepsilon^2\log(r/\varepsilon)}, \qquad f(w)=\rho(w)P(w),\qquad \rho(0)=1,\qquad f(0)=0.
\end{equation*}
Then $f=P$ when $r\leq\varepsilon e^{-1/\varepsilon^2}$ and $f=0$ when $r\geq\varepsilon$. In particular, $f$ is smooth and $D^3f(0)=T$. On the transition annulus, $r^j\abs{D^j\rho}\leq C\varepsilon^2$ for $1\leq j\leq3$. The product rule and homogeneity of $P$ give
\begin{equation}\label{eq:cutoff-estimates}
D^3f=\rho T+O(\varepsilon^2),\qquad \abs{D^jf}\leq Cr^{3-j}\quad(0\leq j\leq2),
\end{equation}
uniformly on $\mathbb R^n$.

For $y=(y_0,y')\ne0$, define the positively homogeneous function of degree two
\begin{equation*}
U(y)=
\begin{cases}
y_0^2f(y'/y_0),&y_0>0,\\
0,&y_0\leq0.
\end{cases}
\end{equation*}
The support of $U$ lies in $\abs{y'}\leq\varepsilon y_0$. Since $U$ vanishes in a neighborhood of every point with $y_0=0$ and $y'\ne0$, it is smooth on $\mathbb R^{n+1}\setminus\set{0}$. Define
\begin{equation}\label{eq:localized-embedding}
F_\varepsilon(x)= \frac{x+i\nabla U(x)}{\sqrt{1+\abs{\nabla U(x)}^2}}, \qquad x\in \mathbb S^n.
\end{equation}
The identity $\operatorname{Re}F_\varepsilon/\abs{\operatorname{Re}F_\varepsilon}=x$ provides a smooth left inverse, proving that $F_\varepsilon$ is an embedding. By Euler's identity $D^2U(x)x=\nabla U(x)$,
\begin{equation*}
F_\varepsilon^*\bar\alpha=\frac{\jkh{x,D^2U\,dx}-\jkh{\nabla U,dx}}{2\kh{1+\abs{\nabla U}^2}}=0,
\end{equation*}
so the embedding is Legendrian.

To estimate the cubic form, pass to the associated Lagrangian cone $\Psi(y)=y+i\nabla U(y)$. In the $y$ coordinates its metric is $G=I+(D^2U)^2$, and its cubic form is $D^3U$. In the support cone, set $w=y'/y_0$, $r=\abs{w}$ and $\ell_a=dy_a-w_a\,dy_0$. Direct differentiation gives
\begin{equation*}
D^3U=y_0^{-1}\sum_{a,b,c=1}^n f_{abc}\,\ell_a\otimes\ell_b\otimes\ell_c, \qquad \abs{\nabla U}\leq Cr^2,\qquad \abs{D^2U}\leq Cr \quad\text{on }\mathbb S^n.
\end{equation*}
Let $u_a$ be the $G$-orthogonal projection of $\partial_{y_a}$ onto the complement of the radial vector $x$, and put $\gamma_{ab}=G(u_a,u_b)$. Then
\begin{equation*}
\ell_a(u_b)=\delta_{ab},\qquad \gamma=I+O(r^2),\qquad \gamma^{-1/2}=I+O(r^2).
\end{equation*}
Uniformity follows from the following estimate: for every $z\in\mathbb R^n$,
\begin{equation*}
z^T\gamma z=\min_sG((0,z)-sx,(0,z)-sx) \geq y_0^2\abs{z}^2\geq\tfrac45\abs{z}^2 \quad(\varepsilon\leq\tfrac12).
\end{equation*}
Set $E_i=\sum_a u_a(\gamma^{-1/2})_{ai}$ and $t=\abs{\Psi(x)}=\sqrt{1+\abs{\nabla U(x)}^2}$. The ambient vectors $d\Psi_x(E_i)$ are orthonormal and perpendicular to the radial direction. By homogeneity they also form an orthonormal tangent frame on the normalized link. The cubic form scales by $t$ under this radial normalization. Thus, in that frame,
\begin{equation*}
\sigma_{ijk}=\frac{t}{y_0} \sum_{a,b,c}f_{abc}(\gamma^{-1/2})_{ai} (\gamma^{-1/2})_{bj}(\gamma^{-1/2})_{ck}.
\end{equation*}
Now $t=1+O(r^4)$, $y_0^{-1}=1+O(r^2)$, and $r\leq\varepsilon$. Equation~\eqref{eq:cutoff-estimates} gives $\sigma=\rho T+O(\varepsilon^2)$. Taking the trace and the squared norm proves \eqref{eq:cubic-localization}. At the north pole, $w=0$, $DU=D^2U=0$ and $D^3f=T$, so the displayed formula gives $\sigma=T$ exactly. Outside the cap, $U$ vanishes and the embedding is totally geodesic.
\end{proof}

For $2\leq p\leq n-2$, $q=n-p$, and $\operatorname{tr}T=0$, Lemma~\ref{lem:cubic-bound} and \eqref{eq:classical-complement} give
\begin{equation}\label{eq:tracefree-separation-bound}
\mathcal L_p^0(T;P) \leq\frac{(p+2)(q+2)}{3(n+2)}\abs{T}^2.
\end{equation}
Since $\mathcal L_p^0$ depends quadratically on the cubic form, \eqref{eq:cubic-localization} also gives, for every $1\leq p\leq n-1$ and uniformly over all planes,
\begin{equation}\label{eq:quadratic-localization}
\mathcal L_p^0(\sigma;P) =\rho^2\mathcal L_p^0(T;P)+O(\varepsilon^2).
\end{equation}

\begin{proof}[Proof of Theorem~\ref{thm:separation}]
Since $F_n(0)>6$ for $n\geq5$ and $4(n-1)/3>6$ when $n\geq6$, choose $b>0$ so that
\begin{equation*}
6<36b^2<F_n(0),\qquad 36b^2<\frac{4(n-1)}3\quad\text{if }n\geq6.
\end{equation*}
Let $T$ be the symmetric cubic tensor whose nonzero components, up to permutation of the indices, are
\begin{equation*}
T_{111}=3b,\qquad T_{122}=b,\qquad T_{133}=T_{144}=-2b.
\end{equation*}
Then $\operatorname{tr}T=0$ and $\abs{T}^2=36b^2$. For $q=n-p$ and $2\leq p\leq n-2$,
\begin{equation*}
pq-\frac{(p+2)(q+2)}{3(n+2)}F_n(0) =\frac{n+2}{2n}(p-2)(q-2)\geq0.
\end{equation*}
Thus \eqref{eq:tracefree-separation-bound} implies $\mathcal L_p^0(T;P)<pq$ for every $p$-plane $P$. Apply Lemma~\ref{lem:cubic-localization}. Since $F_n$ is locally Lipschitz near $0$, the localization estimates give
\begin{equation*}
F_n(\abs{\HH})-\abs{\BB}^2\geq F_n(0)-36b^2-C\varepsilon^2>0
\end{equation*}
for sufficiently small $\varepsilon$. Equations \eqref{eq:tracefree-separation-bound} and \eqref{eq:quadratic-localization} give, uniformly over all planes,
\begin{equation*}
\mathcal L_p^0(\sigma;P) \leq\frac{(p+2)(q+2)}{3(n+2)}\,36b^2+C\varepsilon^2<pq.
\end{equation*}
As there are only finitely many degrees, a single sufficiently small $\varepsilon$ makes all these inequalities strict.

At the north pole, $\sigma=T$ exactly. The Gauss equation gives
\begin{equation*}
R_{1313}=R_{1414}=1-10b^2,\qquad R_{2323}=R_{2424}=1-2b^2,\qquad R_{1234}=0.
\end{equation*}
Hence $\mathcal I=4-24b^2<0$.

For $n\geq6$, it remains to estimate the degree-one trace. Since $T$ is trace free, every unit vector $v$ satisfies
\begin{equation*}
\mathcal L_1^0(T;v) =2\abs{T(v,\cdot,\cdot)}^2-\abs{T(v,v,\cdot)}^2.
\end{equation*}
The tensor $T$ is supported on $E=\operatorname{span}\set{e_1,e_2,e_3,e_4}$. Write $\operatorname{pr}_E v=ru$, where $0\leq r\leq1$ and $u$ is a unit vector in $E$; the case $r=0$ is immediate. Since $\abs{T(u,u,\cdot)}^2\leq\abs{T(u,\cdot,\cdot)}^2$,
\begin{equation*}
\mathcal L_1^0(T;v) =2r^2\abs{T(u,\cdot,\cdot)}^2-r^4\abs{T(u,u,\cdot)}^2 \leq\mathcal L_1^0(T;u).
\end{equation*}
Put $s=u_2^2$ and $c=u_3^2+u_4^2$, so that $u_1^2=1-s-c$. Direct contraction gives
\begin{equation*}
\begin{aligned}
b^{-2}\abs{T(u,\cdot,\cdot)}^2=&18u_1^2+2s+8c,\\
b^{-2}\abs{T(u,u,\cdot)}^2=&(3u_1^2+s-2c)^2+4u_1^2s+16u_1^2c.
\end{aligned}
\end{equation*}
Consequently
\begin{equation*}
b^{-2}\mathcal L_1^0(T;u)=27-24s-6c-9c^2\leq27.
\end{equation*}
The choice of $b$ gives $27b^2<n-1$. Thus \eqref{eq:quadratic-localization} yields the strict inequality in degree one for sufficiently small $\varepsilon$, and $\mathcal L_{n-1}^0(P)=\mathcal L_1^0(P^\perp)$ gives the inequality in degree $n-1$.
\end{proof}

\noindent{\bf Disclosure of AI use}: During the preparation of this manuscript, the authors used ChatGPT to assist with the development and checking of mathematical proofs and to improve the language and presentation. The authors reviewed and revised the AI-assisted material and take full responsibility for the correctness of all mathematical results and proofs, as well as the entire content of the manuscript.

\vspace{0.5cm}
\noindent{\bf Data Availability Statement} This article has no associated data.\\\\
{\bf Conflict of interest} The authors have not disclosed any conflict of interest.


\end{document}